\documentclass[12pt]{amsart}

\usepackage{xypic, color, amssymb}
\newtheorem{theorem}{Theorem}[section]

\newtheorem{proposition}[theorem]{Proposition}
\newtheorem{lemma}[theorem]{Lemma}

\numberwithin{equation}{section}

\begin{document}

\baselineskip=15.5pt

\title[Non-flat extension of flat vector bundles]{Non-flat extension of flat
vector bundles}

\author[I. Biswas]{Indranil Biswas}

\address{School of Mathematics, Tata Institute of Fundamental
Research, Homi Bhabha Road, Bombay 400005, India}

\email{indranil@math.tifr.res.in}

\author[V. Heu]{Viktoria Heu}

\address{IRMA, UMR 7501, 7 rue Ren\'e-Descartes, 67084 Strasbourg Cedex, France}

\email{heu@math.unistra.fr}


\thanks{The first author is supported by the J. C. Bose Fellowship. The second author is
 supported by the ANR grants ANR-13-BS01-0001-01 and ANR-13-JS01-0002-01}

\date{}

\begin{abstract} We construct a pair $(E\, ,F)$, where $E$ is a holomorphic vector bundle
over a compact Riemann surface and $F\, \subset\, E$ a holomorphic subbundle, such
that both $F$ and $E/F$ admit holomorphic connections, but $E$ does not.
\end{abstract}

%\begin{resume}
%Nous construisons une paire $(E\, ,F)$, o\`u $E$ est un fibr\'e vectoriel holomorphe sur une surface de Riemann compacte et $F\, \subset\, E$ est un sous-fibr\'e holomorphe, telles que $F$ et $E/F$ admettent des connexions holomorphes mais $E$ nÕen admet pas. \end{resume}

\maketitle

\section{Introduction}

Let $X$ be a compact connected Riemann surface. Let $E$ be a holomorphic vector bundle over $X$. We say that $E$ is \emph{flat} if it can be endowed with a holomorphic connection. 
Such a holomorphic connection is automatically flat (in the usual sense) 
because there are no nonzero $(2\, ,0)$-forms on $X$.
Conversely, given a $C^\infty$ vector bundle $E$ on $X$, a flat connection
on $V$ defines a holomorphic structure on $E$ as well as a holomorphic connection on it. A criterion due to Atiyah and Weil says that a holomorphic vector bundle $E$ on $X$
is flat if and only if for every holomorphic subbundle $0\,\not=\,F\, \subseteq\, E$
such that there is another holomorphic subbundle $F'\, \subset\, E$
with $F\oplus F'\,=\, E$, the degree of $F$ is zero \cite{At}, \cite{We}. In particular,
any semistable vector bundle on $X$ of degree zero admits a holomorphic connection.

Let $E$ be holomorphic vector bundle on $X$ and $$0\, \subset\, F_1 \, \subset\, F_2\, \subset\, \cdots \, \subset\, F_\ell\,=\, E$$
a filtration by holomorphic subbundles of 
$E$. It is natural to ask for conditions that ensure that $E$ admits a holomorphic
connection that preserves this filtration. An obvious necessary condition is that
each successive quotient $F_i/F_{i-1}$, $1\,\leq\, i \,\leq\, \ell$, should admit
a holomorphic connection. One might expect that this necessary condition
is also sufficient. One reason for this expectation is the following: if $E$ is
semistable of degree zero, then indeed $E$ admits a filtration preserving
holomorphic connection by Simpson correspondence \cite[p. 40, Corollary 3.10]{Si}
(see also \cite[p. 1474]{BH}). Note that if $E$ is semistable of degree zero, and
all successive quotients admit holomorphic connections, then each $F_i$ is
semistable of degree zero.

Our aim here is to show that flatness of vector bundles over curves does not
behave well under extensions. More precisely, we
produce a short exact sequence of holomorphic vector bundles
$$
0\,\longrightarrow\, F\,\longrightarrow\, E\,\longrightarrow\,\mathcal{Q}
\,\longrightarrow\, 0
$$
on any compact Riemann surface $X$ of genus $g\geq 2$ such that both $F$ and $\mathcal{Q}$ admit
holomorphic connections but $E$ does not. 

Note that such a vector bundle cannot exist in  in genus $0$ and $1$. Indeed, in that case, all flat vector bundles are semistable of degree 0, and an extension of a semistable vector bundle of degree $0$ by a semistable vector bundle of degree $0$ is again semistable of degree 0.
The vector bundle $E$ we construct is of rank $3$. Note that this also is a minimal condition since a vector bundle $E$ of rank at most two fitting in an exact sequence as above is
automatically semistable of degree zero (and thus flat). 

\section{Construction of $E$}

Let $X$ be a compact connected Riemann surface of genus $g$, with $g\, \geq\,2$. Denote by $\mathrm{K}_X$ the canonical divisor on $X$. The linear equivalence class of $\mathrm{K}_X$ can be expressed as \begin{equation}\label{kan}\mathrm{K}_X\,=\, P+D\, ,\end{equation} where $P$ is a single point and $D$ is an effective divisor of degree $2g-3$ such that $P$ is disjoint from the support of $D$. Indeed, this follows from the fact that $\mathcal{O}(\mathrm{K}_X)$ is globally generated and $$\dim \mathrm{H}^0(X,\mathcal{O}(\mathrm{K}_X))= g\geq 2.$$
 Fix $P$ and $D$ as in (\ref{kan}). 
Let us now split the divisor $D$ in two parts, namely 
\begin{equation}\label{preimage}
D \,=\, D_Q+D_R ,
\end{equation}
where $D_Q$ and $D_R$ are effective divisors with
\begin{equation}\label{preimagedeg}
\deg (D_Q)+1\,=\,\deg (D_R) \,=\, g-1.
\end{equation} By Serre duality, we know that 
\begin{equation}\label{serre}\mathrm{H}^1
(X,\, \mathcal{O}_X(D_Q+D_R)) \simeq \mathrm{H}^0
(X,\, \mathcal{O}_X(P))\simeq \mathbb{C}.\end{equation}
In particular, we can choose a nonzero element (which is actually unique up to multiplication by a scalar) \begin{equation}\label{theta}
\theta \in \mathrm{H}^1
(X,\, \mathcal{O}_X(D_Q+D_R))\setminus \{0\}.\end{equation}
 Since 
${\mathcal O}_X(D_Q+D_R)\,=\, \mathrm{Hom}_{\mathcal{O}_X}({\mathcal O}_X
(-D_R)\, ,{\mathcal O}_X(D_Q))$, the cohomology class $\theta$ in (\ref{theta}) produces a short exact sequence of vector bundles \begin{equation}\label{V}
0\, \longrightarrow\, {\mathcal O}_X(D_Q)\, \longrightarrow\, V\, \longrightarrow\,{\mathcal O}_X (-D_R)\, \longrightarrow\, 0
\end{equation} on $X$. This exact sequence does not split since $\theta \neq 0$.  From
\eqref{preimagedeg} it follows that $\deg(V)\,=\, -1$.
Consider the holomorphic vector bundle
\begin{equation}\label{E}
E\,:=\, {\mathcal O}_X(P) \oplus V
\end{equation}
on $X$ of rank three and degree zero. We have
\begin{equation}\label{subE}
{\mathcal O}_X(P)\oplus {\mathcal O}_X(D_Q)\, \subset\, E
\end{equation}
because ${\mathcal O}_X(D_Q)\, \subset\, V$ (see \eqref{V}). Let
$s^P$ (respectively, $s^Q$) be the holomorphic section of
${\mathcal O}_X(P)$ (respectively, ${\mathcal O}_X(D_Q)$) given by the
constant function $1$ on $X$. So $s^P$ (respectively, $s^Q$) vanishes
over $P$ (respectively the support of $D_Q$) of order one, and
is nonzero everywhere else. Now consider the holomorphic section
\begin{equation}\label{sigma1}
\sigma_1\, :\, {\mathcal O}_X\, \longrightarrow\,
{\mathcal O}_X(P)\oplus {\mathcal O}_X(D_Q)
\end{equation}
defined by $x\, \longmapsto\, (s^P(x)\, ,s^Q(x))$. Note that $\sigma_1$
does not vanish anywhere because $P$ is disjoint from the
support of $D_Q$ according to (\ref{kan}) and (\ref{preimage}). Let $\sigma$ be the composition
\begin{equation}\label{sigma}
{\mathcal O}_X\, \stackrel{\sigma_1}{\longrightarrow}\,
{\mathcal O}_X(P)\oplus {\mathcal O}_X(D_Q)\, \hookrightarrow\, E
\end{equation}
(see \eqref{subE}). Since $\sigma$ is nowhere vanishing, we get a
short exact sequence of holomorphic vector bundles on $X$
\begin{equation}\label{FQ}
0\, \longrightarrow\, F\,:=\, {\mathcal O}_X\, \stackrel{\sigma}{\longrightarrow}\, E
\, \longrightarrow\, \mathcal{Q}\,:=\, E/\sigma(F)\, \longrightarrow\, 0 \,.
\end{equation}

By construction, ${\mathcal O}_X(P)$ is a direct summand of $E$. Since
$\deg({\mathcal O}_X(P))\, \not=\, 0$, from the criterion
of Atiyah--Weil we conclude that $E$ is not flat. We will now prove that both $F$ and
$\mathcal{Q}$ in \eqref{FQ} are flat.

\subsection{Flatness of $F$ and $\mathcal{Q}$}

\begin{lemma}\label{Lem}
The holomorphic vector bundle $\mathcal{Q}$ in \eqref{FQ} is a nontrivial extension of
the line bundle ${\mathcal O}_X(-D_R)$ by ${\mathcal O}_X(P+D_Q)$.
\end{lemma}

\begin{proof}
On one hand, we have $\bigwedge^2 ({\mathcal O}_X(P)\oplus
{\mathcal O}_X(D_Q)) \,=\, {\mathcal O}_X(P+D_Q)$. On the other hand,
$\bigwedge^2 ({\mathcal O}_X(P)\oplus
{\mathcal O}_X(D_Q)) \,=\, ({\mathcal O}_X(P)\oplus
{\mathcal O}_X(D_Q))/\sigma_1({\mathcal O}_X)$ (see \eqref{sigma1}).
It follows that
\begin{equation}\label{quotientsigma1}
({\mathcal O}_X(P)\oplus
{\mathcal O}_X(D_Q))/\sigma_1({\mathcal O}_X)\,=\, {\mathcal O}_X(P+D_Q)\, .
\end{equation}
The inclusion of ${\mathcal O}_X(P)\oplus {\mathcal O}_X(D_Q)$
in $E$ (see \eqref{sigma}) produces an inclusion of the quotient $({\mathcal O}_X(P)
\oplus {\mathcal O}_X(D_Q))/\sigma_1({\mathcal O}_X)$ in $E/\sigma(F)\,=\,
\mathcal{Q}$ (see \eqref{FQ}). Therefore, from \eqref{quotientsigma1} we have 
\begin{equation}\label{subQ}
{\mathcal O}_X(P+D_Q)\, \subset\,\mathcal{Q}
\end{equation}
as a subbundle. Using \eqref{V}, \eqref{E} we have
$$
\bigwedge\nolimits^2\mathcal{Q}\,=\, \bigwedge\nolimits^3 E\,= {\mathcal O}_X(P)
\otimes \bigwedge\nolimits^2 V\,=\, \,{\mathcal O}_X(P+D_Q-D_R).
$$
Note that its degree is zero \eqref{preimagedeg}. Therefore, from \eqref{subQ},
$$
{\mathcal O}_X(P+D_Q-D_R)\,=\,
\bigwedge\nolimits^2 \mathcal{Q}\,=\, {\mathcal O}_X(P+D_Q)\otimes
(\mathcal{Q}/{\mathcal O}_X (P+D_Q))\, .
$$
So, $\mathcal{Q}/{\mathcal O}_X(P+D_Q)\,=\, {\mathcal O}_X(-D_R)$.
Consequently, from \eqref{subQ}, we get a short exact sequence of vector bundles
\begin{equation}\label{Q}
0\, \longrightarrow\, {\mathcal O}_X(P+D_Q)\,
\longrightarrow\, \mathcal{Q} \, \longrightarrow\,
{\mathcal O}_X(-D_R) \, \longrightarrow\, 0\,.
\end{equation}

To complete the proof of the lemma, we need to show that the short
exact sequence in \eqref{Q} does not split.
Let
\begin{equation}\label{omega}
\omega\, \in\, \mathrm{H}^1(X,\, \mathrm{Hom}({\mathcal O}_X(-D_R)\, ,
{\mathcal O}_X(P+D_Q))
\,=\, \mathrm{H}^1(X,\, {\mathcal O}_X(P+D_Q+D_R))
\end{equation}
be the extension class for the exact sequence in \eqref{Q}. We will now compute
$\omega$.

{}From \eqref{V} and \eqref{E} we have the short exact sequence
$$
0\, \longrightarrow\, {\mathcal O}_X(P)\oplus {\mathcal O}_X(D_Q)\,
\longrightarrow\, E\, \longrightarrow\, {\mathcal O}_X(-D_R)
\, \longrightarrow\, 0
$$
of holomorphic vector bundles on $X$. Let
$$
\begin{array}{rcl}\theta' &\in &\mathrm{H}^1(X,\, ({\mathcal O}_X(P)\oplus
{\mathcal O}_X(D_Q))
\otimes {\mathcal O}_X(D_R))
\vspace{.2cm}\\&&=\, \mathrm{H}^1(X,\,
{\mathcal O}_X(P+D_R))\oplus \mathrm{H}^1(X,\, {\mathcal O}_X(D_Q+D_R))\end{array}
$$
be the cohomology class for this exact sequence. 
Evidently, $\theta'$ coincides with $$(0,\,\theta)\, \in\,
\mathrm{H}^1(X,\,
{\mathcal O}_X(P+D_R))\oplus \mathrm{H}^1(X,\, {\mathcal O}_X(D_Q+D_R))\, ,
$$
where $\theta$ is the class in \eqref{theta}.

Next, consider the homomorphism $\gamma$ defined by the composition
$${\mathcal O}_X(D_Q) \,\hookrightarrow  \,{\mathcal O}_X(P)\oplus {\mathcal O}_X(D_Q)
 \,\twoheadrightarrow \, ({\mathcal O}_X(P)\oplus {\mathcal O}_X(D_Q))/
\sigma_1(F)  \,=\, {\mathcal O}_X(P+D_Q))
%\begin{xy}\xymatrix{
%{\mathcal O}_X(D_Q) \, \, \ar@{^{(}->}[r]\ar@{..>}[d]_\gamma& \,
%{\mathcal O}_X(P)\oplus {\mathcal O}_X(D_Q)\ar@{->>}[d]\\
%{\mathcal O}_X(P+D_Q))& ({\mathcal O}_X(P)\oplus {\mathcal O}_X(D_Q))/
%\sigma_1(F)\ar@2{-}[l]
%}\end{xy}
$$
(see \eqref{quotientsigma1}), where the homomorphism ${\mathcal O}_X(D_Q)\,
\hookrightarrow\,
{\mathcal O}_X(P)\oplus {\mathcal O}_X(D_Q)$ is the inclusion of the second factor.
Clearly, this composition $\gamma$ coincides with the natural inclusion of the
coherent sheaf ${\mathcal O}_X(D_Q)$ in ${\mathcal O}_X(P+D_Q))$. 
Therefore, the cohomology classes $\omega$ and $\theta$ (constructed in \eqref{omega} 
and \eqref{theta}) satisfy the equation
\begin{equation}\label{rho}
\omega \,=\, \rho(\theta)\, ,
\end{equation}
where
$$
\rho\, :\, \mathrm{H}^1(X,\, {\mathcal O}_X(D_Q+D_R))\, \longrightarrow\, 
\mathrm{H}^1(X,\, {\mathcal O}_X(P+D_Q+D_R))
$$
is the homomorphism induced by the natural inclusion of the coherent sheaf\linebreak
${\mathcal O}_X(D_Q+D_R)$ in ${\mathcal O}_X(P+D_Q+D_R)$.
Consider the short exact sequence of coherent sheaves 
$$
0\,\longrightarrow\, {\mathcal O}_X(D_Q+D_R)\,\longrightarrow\,
{\mathcal O}_X(P+D_Q+D_R)
\,\longrightarrow\,{\mathcal O}_X(P+D_Q+D_R)_P\,\longrightarrow\, 0\, ,
$$
where ${\mathcal O}_X(P+D_Q+D_R)_P$ is the torsion sheaf supported at
$P$ with its stalk being the fiber of the line bundle ${\mathcal O}_X(P+D_Q+D_R)$
over $P$. Let
\begin{equation}\label{homology}
\begin{array}{rl}0\,\longrightarrow&\mathrm{H}^0(X,\, {\mathcal O}_X(D_Q+D_R)) \,
\longrightarrow\,\mathrm{H}^0(X,\, {\mathcal O}_X(P+D_Q+D_R))\vspace{.2cm}\\ \,\stackrel{\alpha_1}{\longrightarrow}&
\, {\mathcal O}_X(P+D_Q+D_R)_P 
\,\stackrel{\alpha_2}{\longrightarrow}\, \mathrm{H}^1(X,\, {\mathcal O}_X(D_Q+D_R))\vspace{.2cm}\\
\,\stackrel{\rho}{\longrightarrow}& \mathrm{H}^1(X,\, {\mathcal O}_X(P+D_Q+D_R))
\end{array}
\end{equation}
be the long exact sequence of cohomologies associated to it.
We have
$$\dim \mathrm{H}^0(X,\, {\mathcal O}_X(P+D_Q+D_R))=\dim \mathrm{H}^0(X,\, {\mathcal O}_X(\mathrm{K}_X))=g$$ and, by Riemann-Roch and (\ref{serre}),
$$%\begin{array}{rcl}
\dim \mathrm{H}^0(X,\, {\mathcal O}_X(D_Q+D_R))=g-1.%&= &  2g-3+1-g+\dim \mathrm{H}^1(X,\,{\mathcal O}_X(D_Q+D_R))\vspace{.2cm}\\&= &g-2+\dim \mathrm{H}^0(X,\,{\mathcal O}_X(P))=g-1.
%\end{array}
$$

These imply that $\alpha_1$ in \eqref{homology} is surjective.
Therefore, $\alpha_2$ in \eqref{homology} is the zero homomorphism. This implies that
$\rho$ in \eqref{homology} is injective.

Since $\rho$ is injective, from \eqref{rho} it follows that $\omega\,\not=\, 0$,
because $\theta\,\not=\, 0$ (see \eqref{theta}). The exact sequence in \eqref{Q}
does not split because $\omega\,\not=\, 0$.
\end{proof}

\begin{proposition}\label{Prop}
The holomorphic vector bundle $\mathcal{Q}$ in \eqref{Q} admits a holomorphic connection.
\end{proposition}

\begin{proof}
Assume that $\mathcal{Q}$ does not admit any holomorphic connection. Since
$\text{degree}(\mathcal{Q})
\,=\, 0$, and $\mathcal{Q}$ does not admit any holomorphic connection, the criterion
of Atiyah--Weil says that $\mathcal{Q}$ holomorphically decomposes as
\begin{equation}\label{decQ}
\mathcal{Q}\, =\, L\oplus M\, ,
\end{equation}
where $\text{degree}(L)\,=\, -\text{degree}(M)\, >\, 0$. Let
$
p_M\, :\, \mathcal{Q}\, \longrightarrow\, M
$
be the projection given by the decomposition in \eqref{decQ}.
Let $\beta$ denote the composition
$$
{\mathcal O}_X(P+D_Q)\, \hookrightarrow\,
\mathcal{Q}\, \stackrel{p_M}{\longrightarrow}\, M\, ,
$$
where the inclusion is constructed in \eqref{subQ}. Since $$\text{degree}({\mathcal
O}_X(P+D_Q))\,=\, g-1>\,0\, >\,
\text{deg}(M)\, ,$$ there is no nonzero homomorphism from
${\mathcal O}_X(P+D_Q)$ to $M$. In particular, $\beta\, =\, 0$. 

We have ${\mathcal O}_X(P+D_Q)\,\subset\, L$ because $\beta\,=\, 0$. Since both
${\mathcal O}_X(P+D_Q)$ and $L$ are line subbundles on $\mathcal{Q}$, this
implies that the two subbundles ${\mathcal O}_X(P+D_Q)$ and $L$ coincide. Hence
$$
M\, =\,\mathcal{Q}/L\,=\,\mathcal{Q}
/{\mathcal O}_X(P+D_Q)\,=\,{\mathcal O}_X(-D_R)
$$
(see Lemma \ref{Lem}).
Therefore, the decomposition $\mathcal{Q}\,=\,L\oplus M$ in \eqref{decQ} produces
a splitting of the short exact sequence in \eqref{Q}. But we know from
Lemma \ref{Lem} that the short exact sequence in \eqref{Q} does not split.
In view of the above contradiction we conclude that $\mathcal{Q}$ admits
a holomorphic connection.
\end{proof}

As we have seen, $E$ is not flat by construction. 
On the other hand, consider the short exact sequence in \eqref{FQ}. The trivial
holomorphic line bundle $F\,=\,{\mathcal O}_X$ admits the trivial holomorphic connection.
The quotient bundle $\mathcal{Q}$ is flat by Proposition \ref{Prop}.
Therefore, we have the following:

\begin{theorem}\label{Thm}
Let $X$ be a compact connected Riemann surface of genus $g\geq 2$. The vector bundle $E$ in \eqref{E}  has a holomorphic subbundle 
such that both the subbundle and the quotient bundle admit holomorphic connections.
But $E$ does not admit a holomorphic connection.
\end{theorem}

%%%%%%%%%%%%%%%%%%%%%%%%%%%%%%%%%%%%%%%%%%%%%%%%%%%%%%%%%%%%%%

\end{document}